\documentclass[10pt,a4paper]{amsart}

\usepackage[utf8]{inputenc}
\usepackage{t1enc}

\usepackage{stmaryrd}

\usepackage{amssymb,amsfonts}
\usepackage{nccmath}
\usepackage{mathtools}

\usepackage{xcolor}

\usepackage[margin=2.7cm,marginpar=2cm]{geometry}

\usepackage[unicode]{hyperref}

\usepackage[capitalize,nameinlink]{cleveref}
\numberwithin{equation}{section}

\def\equationautorefname~#1\null{Equation~(#1)\null}

\def\={\,=\,}

\usepackage{thmtools}
\usepackage{enumitem}

\makeatletter
\def\@endtheorem{\endtrivlist}
\makeatother

\newlist{thmlist}{enumerate}{1}
\setlist[thmlist]{label=(\roman{thmlisti}),
	ref=\theThm(\roman{thmlisti}),
	noitemsep}

\declaretheorem[
style=plain,
name=Theorem,
numbered=yes,
refname={Theorem,Theorems},
Refname={Theorem,Theorems}
]{Thm}

\declaretheorem[
style=definition,
name=Remark,
numberlike=Thm,
refname={Remark,Remarks},
Refname={Remark,Remarks}
]{Rem}

\Crefname{Thm}{Theorem}{Theorems}
\Crefname{listthm}{Theorem}{Theorems}

\addtotheorempostheadhook[Thm]{\crefalias{thmlisti}{listthm}}

\newcommand{\ii}{\mathrm{i}}
\newcommand{\dd}{\mathrm{d}}

\newcommand{\sage}{\texttt{SAGE}}
\newcommand{\pari}{\texttt{PARI/GP}}

\makeatletter
\@namedef{subjclassname@2020}{%
	\textup{2020} Mathematics Subject Classification}
\makeatother

\newcommand{\pmodi}[1]{\mkern4mu(\mathrm{mod} \, #1)}
\newcommand{\kr}[2]{\genfrac(){1pt}{1}{#1}{#2}}
\newcommand{\krd}[2]{\genfrac(){1pt}{0}{#1}{#2}}
\DeclareMathOperator{\SL}{SL}

\usepackage{listings}
\lstdefinestyle{mystyle}{
	basicstyle=\ttfamily\footnotesize,
	breakatwhitespace=false,         
	breaklines=true,                 
	keepspaces=true,                 
	numbers=none,                    
	showspaces=false,                
	showstringspaces=false,
	showtabs=false,                  
	tabsize=2
}
\lstdefinelanguage{gp}{
	basicstyle=\color{gray},
	keywords = {for, if, vector, Mat, vecsort, vecmax, apply, concat, sum, error, vecsearch, lift, subst, variable, serprec},
	keywordstyle=\color{blue},
	string=[d]",
	stringstyle=\color{codepurple},
	comment=[l]{\\\\},
	commentstyle=\color{green!50!black}\ttfamily,
	basicstyle=\ttfamily\footnotesize,
	numbers=left,
	numberstyle=\tiny\color{gray!50!white},
	breakatwhitespace=false,         
	breaklines=true,                 
	captionpos=b,                    
	keepspaces=true,                 
	numbers=left,                    
	numbersep=5pt,                  
	showspaces=false,                
	showstringspaces=false,
	showtabs=false,                  
	tabsize=2
}
\lstdefinelanguage{gpin}{ 
	basicstyle=\ttfamily\footnotesize,
	comment=[l]{\\},
	commentstyle=\color{green!50!black}\ttfamily,
	belowskip=-5pt
}
\lstdefinelanguage{gpout}{
	basicstyle = \color{blue}\ttfamily\footnotesize,
}
\lstdefinelanguage{gpoutmsg}{
	basicstyle = \color{black}\ttfamily\footnotesize,
}

\definecolor{codegreen}{rgb}{0,0.6,0}
\definecolor{codegray}{rgb}{0.5,0.5,0.5}
\definecolor{codepurple}{rgb}{0.58,0,0.82}

\lstdefinestyle{sage}{
	commentstyle=\color{codegreen},
	keywordstyle=\color{magenta},
	numberstyle=\tiny\color{gray!50!white},
	stringstyle=\color{codepurple},
	basicstyle=\ttfamily\footnotesize,
	breakatwhitespace=false,         
	breaklines=true,                 
	captionpos=b,                    
	keepspaces=true,                 
	numbers=left,                    
	numbersep=5pt,                  
	showspaces=false,                
	showstringspaces=false,
	showtabs=false,                  
	tabsize=2
}
\lstdefinelanguage{sagein}{ 
	basicstyle=\ttfamily\footnotesize,
	belowskip=-5pt
}
\lstdefinelanguage{sageout}{
	basicstyle = \color{blue}\ttfamily\footnotesize,
}

\usepackage{accsupp}

\makeatletter
\def\paragraph{\@startsection{paragraph}{4}%
	{0em}{1ex \@plus 0.5ex \@minus 0.25ex}{-\fontdimen2\font}%
	{%
		\ttfamily
}}
\makeatother

\begin{document}
	
	\title[Relations for the partition function mod 4]{Explicit linear dependence congruence relations for \\ the partition function modulo 4}
	
       \author[Steven Charlton]{Steven Charlton}
		\address{Max Planck Institute for Mathematics \\ Vivatsgasse 7 \\ 53111 Bonn \\ Germany}
	\email{charlton@mpim-bonn.mpg.de}

	

	\begin{abstract} 
				Almost nothing is known about the parity of the partition function $p(n)$, which is conjectured to be random.  Despite this expectation, Ono \cite{ono2022pmod4} surprisingly proved the existence of infinitely  many linear dependence congruence relations modulo 4 for $p(n)$, indicating that the parity of the partition function cannot be truly random.  Answering a question of Ono, we explicitly exhibit the first examples of these relations which he proved theoretically exist.  The first two relations invoke 131 (resp. 198) different discriminants \( D \leq 24 k - 1 \) for \( k = 309 \) (resp. \( k = 312 \)); new relations occur for \( k = 316,\, 317,\, 319,\, 321,\, 322,\, 326, \, \ldots  \).
	\end{abstract}

	\leavevmode
	\maketitle

\section{Introduction}

	A \emph{partition} of a non-negative integer \( n \) is a non-increasing sequence \( \lambda_1 \geq \lambda_2 \geq \cdots \geq \lambda_k \) of positive integers, such that \( n = \lambda_1 + \lambda_2 + \cdots + \lambda_k \).  The number of partitions of \( n \) is denoted by \( p(n) \).  For example, there are \( p(4) = 5 \) partitions of 4, namely
	\[
		4 \= 3 + 1 \= 2 + 2 \= 2 + 1 + 1 \= 1 + 1 + 1 + 1 \,.
	\]
	The partition function \( p(n) \) is known to satisfy many congruences, perhaps most famously are Ramanujan's \cite{ramanujan1920congruence,berndt1999ramanujan} congruences
	\begin{align*}
		p(5n+4) & \equiv 0 \pmod{5} \,, \\
		p(7n+5) & \equiv 0 \pmod{7} \,, \\
		p(11n+7) & \equiv 0 \pmod{11} \,.
	\end{align*}
	These congruences have inspired much further work, culminating in particular with the result of Ahlgren and Ono \cite{ahlgren01congruence} that there are infinitely many congruences for all moduli coprime to 6.
	
	In contrast to this rich theory of congruences, almost nothing is known about the partition function modulo 2 or 3.  The general expectation is that parity of \( p(n) \) is ``random''; more precisely Parkin and Shanks \cite{parkin1967parity} conjecture that odd and even parities are `equi-distributed'
	\[
	\lim_{X\to\infty} \frac{\{ n \leq X \mid \text{$p(n)$ is even (resp. odd)} \}}{X} = \frac{1}{2} \,.
	\]
	(And, moreover, that the binary number whose $k$-th binary place encodes the parity of \( p(k) \) should be \emph{normal} in base 2.)  	By an elementary Euclid-style argument, Kolberg~\cite{kolbeg52parity} already established that infinitely many partition values are odd, and infinitely many are even, without specifying which.
	
	Ono \cite{ono1996parity} showed that in any arithmetic progression \( r \pmodi{t} \), there are infinitely many \( N \equiv r \pmod{t} \) with \( p(N) \) even, and that there are infinitely many \( M \equiv r \pmod{t} \) with \( p(M) \) odd, provided there is at leat one such \( M \).  Moreover, Ono \cite{ono1996parity} gave an explicit bound for the smallest \( M \equiv r \pmod{t} \) with \( p(M) \) odd, if any such \( M \) exists.  Radu \cite{radu12subbarao} completed Ono's strategy by showing the initial odd values do exist, making Ono's result on odd values \( p(M) \) in arithmetic progressions unconditional.  Radu \cite{radu12subbarao} also showed that for any arithmetic progression \( r \pmod{t} \), there is \( L \equiv r \pmod{t} \) for which \( 3 \nmid p(L) \).  In particular Radu \cite{radu12subbarao} showed there are no congruences of the form \( p(an + b) \equiv 0 \pmodi{\ell} \), for \( \ell \in \{ 2, 3 \} \).  
	
	For each square-free positive integer \( D \equiv 23 \pmodi{24} \), Ono was able to prove \cite{ono2010parity} that in the quadratic families, \( {p(\frac{Dm^2+1}{24})} \) is infinitely often even (resp. odd), by capturing the parity of \( {p(\frac{Dm^2+1}{24})} \) via modular forms.  Refining this: for  each square-free positive integer \( D \equiv 23 \pmodi{24} \), Ono considered the twisted generating function
	\[
		P(D; q) \coloneqq \sum_{m,n\geq1} \chi_{-D}(n) \chi_{12}(m) p\Big( \frac{Dm^2+1}{24} \Big) q^{n \, m} \,,
	\]
	where \(  \chi_{d}(n) \coloneqq \kr{d}{n} \) is the Kronecker symbol of discriminant \( d \) .  He showed \cite[Theorem 1.1]{ono2022pmod4} that \( P(D; q) \) is congruent modulo 4 to the Fourier expansion at \( \ii \infty \) (where \( q \coloneqq e^{2 \pi \ii \tau} \)) of some weight 2 meromorphic modular form of level \( \Gamma_0(6) \).  (Explicitly, as given in \cite[Section 3]{ono2022pmod4}, to a logarithmic derivative \( \mathcal{L}_D(\tau) \) of a generalised Borcherds product \( \Psi_D(\tau) \).)
	
	Now fix a set \( S \) of positive square-free \( D \equiv 23 \pmodi{24} \), and write \( h_S = \max\{ h(-D) \mid D \in S \} \) to be of the maximum of the class numbers \( h(-D) \) of the associated imaginary quadratic fields \( \mathbb{Q}(\sqrt{-D}) \).  As will be recalled in \autoref{sec:review}, the holomorphic normalisation  \( \widehat{P}_S(D; q) \) relative to the set \( S \), of the generating series \( P(D;q) \), is a weight \( 12 h_S + 2 \) holomorphic modular form on \( \Gamma_0(6) \).  By Sturm \cite{sturm1987congruence} two such modular forms are congruent if their Fourier are congruent for the first 
	\[
		[\SL_2(\mathbb{Z}) \colon \Gamma_0(6)] \cdot (12 h_S + 2) / 12 = 12 h_S + 2 
	\]
	many terms.  Since the class number \( h(-D) \) grows at most like \( \sqrt{D} \log(D) \), while the set of positive square-free \( D \equiv 23 \pmodi{24} \) has positive natural density, one can necessarily construct a set \( S \) with \( \#S > 12 h_S + 2 \) (indeed infinitely many disjoint sets), as we will recall in \autoref{sec:search}.  This guarantees the existence of linear dependence congruence relations modulo 4 between \( \widehat{P}_S(D; q) \), whence the partition function is `not truly random modulo 4'. \medskip
	
	The goal of this note is to answer a question of Ono about these linear dependence congruence relations modulo 4, by explicitly the first few such relations (one of which is in fact already a relation modulo 2).  In particular, we have the following result (where the schematic identities in \cref{thm:main1,thm:main2} are given explicitly in \autoref{sec:explicit}).
	\begin{Thm}\label{thm:main-schematic}
		Order identities between \( \smash{\widehat{P}_S}(D; q) \) according to the maximal discriminant appearing, then:
		\begin{thmlist}[itemsep=1ex]
		\item \label{thm:main1} The first linear dependence congruence relation modulo 4 between \( \widehat{P}_{S^{(1)}}(D; q) \) occurs at \( D = 24 \cdot 309 - 1 \), amongst a set \( S^{(1)} = \{ 24k - 1 \mid k \in \{ 2, 4, 6, \ldots, 309 \} \} \) of cardinality 131 and maximal class number \( h_{S^{(1)}} = 92 \).
		\item\label{thm:main2} The second linear dependence congruence relation modulo 4 between \( \widehat{P}_{S^{(2)}}(D; q) \) occurs at \( D = 24 \cdot 312 - 1 \), amongst a set \( S^{(2)} = \{ 24k - 1 \mid k \in \{ 2, 4, 6, \ldots, 312 \} \} \) of cardinality 198 and maximal class number still \( h_{S^{(2)}} = 92 \).
		\item\label{thm:main3} Including the first two identities, new identity appears at each \( D = 24k - 1 \), for
		\[
			k = 309,\, 312,\, 316,\, 317,\, 319,\, 321,\, 322,\, 326,\, 327,\, 332,\, 336,\, 337,\, \ldots
		\]
		\end{thmlist}
	\end{Thm}

	The proof of this theorem is purely computational in nature.  For \cref{thm:main1,thm:main2} we compute each series to sufficiently high order using \pari{} \cite{PARI2}, and check the linear combinations vanish to this order.  The necessary order is given by the Sturm bound, here \( 12 h_{S^{(i)}} + 2 = 1106 \) in both cases.  For \cref{thm:main3} we compute all series up to \( D \leq 24\cdot350 - 1 \) (say), to the precision required by the Sturm bound: the maximum class number \( h(-D) \) of the square-free positive \( D \equiv 23 \pmod{24} \) up to \( D \leq 24\cdot350-1 \) is \( h_{\leq350} = 134 \), with associated Sturm bound \( 12 h_{\leq 350} + 2 = 1610 \).  The identities are given by solving the resulting modulo 4 linear system using the \pari{} routine \texttt{matsolvemod}, from which one can read off the necessary new discriminants at each step.  (The identities in \cref{thm:main1,thm:main2} are, of course, part of this solution space.)\medskip
	
	In \autoref{sec:review} we recall the details of the construction of the holomorphic normalisation \( P_S(D; q) \), and indicate how to compute it in \pari{} \cite{PARI2}.
	In \autoref{sec:search} we recall Ono's argument that a linear dependence congruence relation modulo 4 exists, explicitly identify the search space, and reduce it a \emph{tractable} computation.
	Finally in \autoref{sec:explicit} we give the explicit identities found during our search.  
		
	\subsection*{Computational resources.}  The initial computations for this note took place on an  {\small \texttt{Intel(R) Xeon(R) Gold 6148 CPU @ 2.40GHz}} machine with 160 threads, and 3TB RAM.  After discovering the shortest identities involved only \( D \leq 24 \cdot 312 - 1 \), the computations were revised to be practical on a single thread, with limited RAM.  The relevant \sage{} and \pari{} code to find all identities up to \( D \leq 24 \cdot 350 - 1 \) (using precision \( 12 h_S + 2 + 10 \)) is given \autoref{sec:code}.  The code is attached to the arXiv submission, along with a precomputed table of the necessary partition values in \texttt{values\_part350.txt}.
	
	\subsection*{Acknowledgements.} I am grateful to the Max Planck Institute for Mathematics, Bonn, for support, hospitality, computing resources and excellent working conditions during the preparation of this note.  I am also grateful to the organisers (Kathrin Bringmann, Walter Bridges, Johann Franke) of the \href{http://www.mi.uni-koeln.de/Bringmann/Konferenz2024.html}{International Conference on Modular Forms and $q$-Series} (11--15 March 2024, Universit\"at zu K\"oln) for the chance to learn more about partitions, and for catalysing this work via talks of and with Ken Ono.

	\section{Review of Borcherds products and holomorphic normalisations}\label{sec:review}
	
	The first main result (Theorem 1.1) of \cite{ono2022pmod4} is that for square-free positive \( D \equiv 23 \pmodi{24} \) the twisted generating series
	\[
		P(D; q) \coloneqq \sum_{m,n\geq1} \chi_{-D}(n) \chi_{12}(m) p\Big( \frac{Dm^2+1}{24} \Big) q^{n \, m} \,,
	\]
	is congruent modulo 4 to a weight 2 meromorphic modular form \( \mathcal{L}_D(\tau) \) on \( \Gamma_0(6) \).  The proof \cite[Sections 2 and 3]{ono2022pmod4} gives an explicit construction for \( \mathcal{L}_D(\tau) \), which we briefly recall. \medskip
	
	Write \( P(q) \) for the generating series of partition values
	\[
	P(q) \coloneqq \sum_{n=0}^\infty p(n) q^n = \prod_{m=1}^\infty (1 - q^m)^{-1} = 1 + q + 2q^2 + 3q^3 + 5q^4 + O(q^5)\,.
	\]
	The third-order mock theta functions \( f(q) \) and \( \omega(q) \) are defined by
	\begin{align*}
		f(q) \coloneqq {} & 1 + \sum_{n=1}^\infty \frac{q^{n^2}}{(1 + q)^2 (1 + q^2)^2 \cdots (1 + q^n)^2 } =   1 + q - 2q^2 + 3q^3 - 3q^4 + O(q^5) \\[1ex]
		\omega(q) \coloneqq {} & \sum_{n=1}^\infty \frac{q^{2n(n+1)}}{(1 - q)^2 (1 - q^3)^2 \cdots (1 - q^{2n+1})^2 } = 1 + 2q + 3q^2 + 4q^3 + 6q^4 + O(q^{5})
	\end{align*}
	Note: the coefficients of \( f(q) \) and \( \omega(q) \) are integral, documented in OEIS sequences \href{https://oeis.org/A000025}{A000025}, and  \href{https://oeis.org/A053253}{A053253} respectively.  Define the vector-valued function \( R(\tau) = (R_0(\tau), \ldots, R_{11}(\tau)) \), with components
	\[
		R_k(\tau) = \begin{cases}
			0 & \text{if $k=0,3,6,9$} \\
			\chi_{-12}(k) q^{-1} f(q^{24}) & \text{if $k = 1,5,7,11$} \\
			2q^8(-\omega(q^{12}) + \omega(-q^{12})) & \text{if $j=2$} \\
			2q^8(-\omega(q^{12}) - \omega(-q^{12})) & \text{if $j=4$} \\
			2q^8(\omega(q^{12}) + \omega(-q^{12})) & \text{if $j=8$} \\
			2q^8(\omega(q^{12}) - \omega(-q^{12})) & \text{if $j=10$} \,.
		\end{cases}
	\]
	An important observation \cite[Lemma 3.2]{ono2022pmod4} is that \( P(q) \equiv f(q) \pmodi{4} \).  Together with the fact that \( 2 ( \pm \omega(q^{12}) \pm \omega(-q^{12})) \equiv 0 \pmodi{4} \), one directly has
	\begin{equation}\label{eqn:rmod4}
		R_k(\tau) \equiv \begin{cases}
			\chi_{-12}(k) q^{-1} f(q^{24})  \equiv \chi_{-12}(k) q^{-1} P(q^{24}) \pmodi{4} & \text{if $k=1,5,7,11$} \\
			0 \pmodi{4} & \text{otherwise} \,.
		\end{cases}
	\end{equation}
	Write 
	\(
		R_k(\tau) \eqqcolon \sum_{n \gg -\infty} C_R(j;n) q^n \,,
	\)
	for the coefficients of the \( R_k(\tau) \) component.
	
	For square-free positive \( D \equiv 23 \pmodi{24} \), the generalised Borcherds product \( \Psi_D(\tau) \) is defined by
	\[
		\Psi_D(\tau) \coloneqq \prod_{m=1}^\infty \psi_D(q^m)^{C_R(\overline{m}; Dm^2)} \,, \quad 
	\text{ where } \quad 
		\psi_D(X) \coloneqq \prod_{\text{$b$ mod $D$}} (1 - e^{-2\pi \ii b / D})^{\text{$\kr{-D}{b}$}} \,,
	\]
	and \( \overline{m} \) denotes the canonical residue class of \( m \) modulo 12.  Theorem 2.1 \cite{ono2022pmod4} (referring to Section 8.2 of \cite{bruinier2010heegner}), establishes that \( \Psi_D(\tau) \) is a weight 0 meromorphic modular form on \( \Gamma_0(6) \).
	
	Finally recall the holomorphic differential operator \( \mathbb{D} \coloneqq q \frac{\dd}{\dd q} = \frac{1}{2\pi\ii} \frac{\dd}{\dd \tau} \).  Define
	\[
		\mathcal{L}_D(\tau) \coloneqq \frac{1}{-\sqrt{-D}} \frac{(\mathbb{D} \Psi_D)(\tau)}{\Psi_D(\tau)} \,.
	\]
	Then Lemma 3.1 \cite{ono2022pmod4} establishes that \( \mathcal{L}_D(\tau) \) is a weight 2 meromorphic modular form on \( \Gamma_0(6) \) (as \( \mathbb{D} \) agrees with the Ramanujan-Serre  derivative in weight 0), and that the poles of\( \mathcal{L}_D(\tau) \) are simple and located at the \( \Gamma_0(6) \) Heegner points with discriminant \( -D \). Lemma 3.1 \cite{ono2022pmod4} also calculates the following \( q \)-expansion:
	\[
		\mathcal{L}_D(\tau) = \sum_{m=1}^\infty C_R(\overline{m}; Dm^2) m \cdot \!\!\! \sum_{\substack{n=1 \\ \gcd(n,D)=1}}^\infty \!\!\! \krd{-D}{n} q^{m \, n} \,.
	\]
	It now follows directly from \autoref{eqn:rmod4} and the congruence \( \chi_{-12}(m) m \equiv \chi_{12}(m) \pmodi{4} \), that
	\[
		\mathcal{L}_D(\tau) \equiv P(D; q) \pmod{4} \,.
	\]
	So \( P(D; q) \) is indeed congruent modulo 4 to a weight 2 meromorphic modular form, the idea now is to rescale this to obtain a \emph{holomorphic normalisation} in a finite dimensional space.
	\medskip
	
	Recall \( \Delta(\tau) \) (the modular discriminant) is the unique normalised weight 12 cusp form for \( \SL_2(\mathbb{Z}) \); its $q$-expansion starts
	\[
	\Delta(\tau) =  q - 24q^2 + 252q^3 - 1472q^4 + 4830q^5 + O(q^{6}) \,.
	\]
	For fundamental discriminants \( -D < 0 \), (in particular for \( D \in S \)), define the \emph{Hilbert class polynomial} \( H_{-D}(X) \) as follows,
	\[
	H_{-D}(X) \coloneqq \prod_{Q \in \mathcal{Q}_{-D} / \operatorname{PSL}_2(\mathbb{Z})} \hspace{-1em} (X - j(\tau_Q)) \in \mathbb{Z}[X] \,.
	\]
	Here
	\begin{itemize}[itemsep=1ex,left=1em]
		\item \( \mathcal{Q}_{-D} /  {\operatorname{PSL}_2(\mathbb{Z})} \) is the set of reduced (necessarily primitive when \( -D \) is fundamental) positive-definite integral binary quadratic forms \( ax^2 + bxy + cy^2 \) of discriminant \( b^2 - 4ac = -D \), 
		\item \( \tau_Q \coloneqq \frac{-b + \sqrt{-D}}{2a} \in \mathbb{H} \) is the Heegner point of a quadratic form \( Q = ax^2 + bxy + cy^2 \), 
		\item  \( j(\tau) \coloneqq  E_4(\tau)^3 / \Delta(\tau) \) is the elliptic $j$-invariant, a hauptmodul for the field of modular functions on \( \SL_2(\mathbb{Z}) \), whose $q$-expansion starts 
		\[ 
		j(\tau)  = q^{-1} + 744 + 196884q + 21493760q^2 +  864299970q^3 + O(q^4) \,, 
		\]
		and
		\item \( E_4(\tau) \coloneqq 1 + 240 \sum_{n=1}^\infty \sum_{d \mid n} d^3 q^n = 1 + 240q + 2160q^2 + 6720q^3 + O(q^4) \) is the weight 4 Eisenstein series.
	\end{itemize}
	In particular, the degree of \( H_{-D}(X) \) is the class number \( h(-D) \).
	
	\begin{Rem}
		The Hilbert class polynomial \( H_{-D}(X) \) can be computed in \pari{} using the command \texttt{polclass(-D)}.  The resulting polynomial generates the Hilbert class field for the imaginary quadratic field \( \mathbb{Q}(\sqrt{-D}) \).
	\end{Rem}
	
	Continuing as in \cite[Section 4]{ono2022pmod4}: since the poles of \( \mathcal{L}_D(\tau) \) are simple, and located at the discriminant \( -D \) Heegner points, it follows 
	\[
		\mathcal{L}_D(\tau) \cdot \Delta(\tau)^{h(-D)} \cdot H_{-D}(j(\tau)) 
	\]
	is a wight \( 12 h(-D) +2 \) \emph{holomorphic} modular form on \( \Gamma_0(6) \).  Finally, since \( j(\tau) = E_r(\tau)^3 / \Delta(\tau) \equiv 1 / \Delta(\tau) \pmodi{4} \), we have that
	\begin{equation}\label{eqn:lnorm}
		\mathcal{L}_D(\tau) \cdot \Delta(\tau)^{h(-D)} \cdot H_{-D}(j(\tau))  \equiv P(D;q) \cdot \Delta(\tau)^{h(-D)} \cdot H_{-D}(1/\Delta(\tau)) \pmod{4} \,.
	\end{equation}
	Multiplying by further factors of \( \Delta(\tau) \) increases the weight, but still gives holomorphic modular forms on \( \Gamma_0(6) \), allowing us to find linear dependence congruence relations modulo 4, as follows. \medskip
	
	Let \( S \) be a finite set of square-free positive integers \( D \equiv 23 \pmod{24} \).  Write \( h_S \coloneqq \max\{ h(-D) \mid D \in S \} \) to be the maximum of the class numbers \( h(-D) \), ranging over \( D \in S \).  The \emph{holomorphic normalisation} of \( P(D; q) \), relative to the set \( S \), is defined to be
	\[
		\widehat{P}_S(D; q) \coloneqq P(D; q) \cdot \Delta(\tau)^{h_S} \cdot H_{-D}(1 / \Delta(\tau)) \,.
	\]
	By \autoref{eqn:lnorm}, each of these is congruent modulo 4 to a holomorphic modular form of weight \( 12 h_S + 2 \) on \( \Gamma_0(6) \).  Hence we obtain:
	
	\begin{Thm}[Theorem 1.2, \cite{ono2022pmod4}]\label{thm:ono4}
			If \( \# S > 12 h_S + 2 \), then the \( q \)-series \( \{ \widehat{P}_S(D; q) \mid D \in S \} \) are linearly dependent modulo 4.
	\end{Thm}

	\begin{proof}[Proof {\normalfont (Section 4, \cite{ono2022pmod4})}]
		A theorem of Sturm \cite{sturm1987congruence} establishes that two weight \( k = 12 h_S + 2 \) holomorphic modular of level \( \Gamma_0(6) \) are congruent if their coefficients agree for all indices up to (and including)
		\[
		[ \SL_2(\mathbb{Z}) \colon \Gamma_0(6)] \cdot \frac{k}{12} = 12 h_S + 2 \,.
		\]
		As \( \#S > 12 h_S + 2\), there is necessarily a linear combination of the holomorphic normalisations \( \widehat{P}_S(D; q) \), whose first \( 12h_S + 2 \) coefficients vanish modulo 4 (by construction the constant coefficient of \( P(D;q) \), and hence \( \widehat{P}_S(D;q) \) vanishes, removing one degree of freedom).  This forces the linear combination to vanish identically modulo 4, giving the required linearly dependence.
	\end{proof}

	The punchline is now as follows.  Since the class number \( h(-D) \), for fundamental discriminants \( -D < 0 \), satisfies the inequality \cite[Lemma 2.2]{griffin2021elliptic},
	\begin{equation}\label{eqn:classno}
		h(-D) \leq \frac{\sqrt{D}}{\pi} (\log{D} + 2) \,,
	\end{equation}
	but the natural density of square-free integers \( D \equiv 23 \pmod{24} \) is positive, one is eventually guaranteed to find a set \( S \) satisfying the hypothesis of \autoref{thm:ono4}.  We make this precise in the next section.

	\section{Defining and refining the search space}
	\label{sec:search}

	Consider the set
	\[
	S(t) \coloneqq \{ 0 < D \leq t \mid \text{$D \equiv 23 \pmodi{24}$, and $D$ square-free} \} \,,
	\]
	of all square-free positive integers \( D \equiv 23 \pmodi{24} \) up to \( t \) inclusive.  (The negative of) each of these is a fundamental discriminant, and since the bound on \( h(-D) \) in \autoref{eqn:classno} is monotonically increasing, we have
	\[
		h_{S(t)} \leq \frac{\sqrt{t}}{\pi} (\log{t} + 2) \,.
	\]
	
	Landau \cite[\S174,pp.~633--636]{landau53handbuch}, on the other hand, established an asymptotic counting the number of square-free integers in a given residue class.  Let \( Q(x; k, \ell) \) denote the number of square-free integers \( N \leq x \) with \( N \equiv \ell \pmod{k} \).  In the special case \( \gcd(\ell,k) = 1 \), Landau \cite[p.~635]{landau53handbuch} first established:
	\[
		\lim_{x \to \infty} \frac{Q(x; k,\ell)}{x} = \frac{1}{\zeta(2)} \cdot \frac{1}{k}  \prod_{\substack{\text{$p$ prime} \\ p \mid k}} \frac{1}{1 - p^{-2}} \,.
	\]
	(He handled the case \( \gcd(\ell,k) \geq 1 \) afterwards \cite[p. 636]{landau53handbuch}, using this.)  Observe that \( \# S(t) = Q(t; 24, 23) \) with \( \gcd(23,24) = 1 \), so we obtain from Landau's formula that
	\[ 
		\lim_{t \to \infty} \frac{\# S(t)}{t} =
		\frac{1}{\zeta(2)} \cdot \frac{1}{24} \prod_{\substack{\text{$p$ prime} \\ {p \mid 24}}} \frac{1}{1-p^{-2}}  =  \frac{1}{16\zeta(2)} = 0.0379954438\ldots > 0 \,.
	\]
	In particular the set \( S(t) \) has positive density, containing approximately 3.8\% of the positive integers, and \( \# S(t) \sim \frac{t}{16\zeta(2)} \), which grows faster than \( h_{S(t)} = \frac{\sqrt{t}}{\pi} (\log{t} + 2) \).  Hence eventually \( 12 h_{S(t)} + 2 < \# S(t) \) holds for sufficiently large \( t \), 
	and one necessarily obtains a linear dependence congruence relation modulo 4 amongst the \( \{ \widehat{P}_{S(t)}(D; q) \mid D \in S(t) \} \).  $\llbracket$As Ono notes in the remarks under \cite[Theorem 1.2]{ono2022pmod4}, this  establishes the unconditional \emph{existence} of such linear dependence congruence relations modulo 4.$\rrbracket$\medskip
	
	We now identify a tractably small set \( S \), which we can computationally investigate, in order to find the first such linear dependence congruence relation  modulo 4 for the partition function.
		
	\paragraph{First estimate for \( t \)} As a first estimate for the necessary \( t \), with that \( 12 h_{S(t)} + 2 < \#S(t) \), which guarantees a linear dependence via \autoref{thm:ono4}, solve the inequality 
	\[
		12 \frac{\sqrt{t}}{\pi} (\log{t} + 2) + 2 \leq \frac{t}{16\zeta(2)} \,,
	\]
	This can be done numerically with a variety of techniques\footnote{say Brent's method as implemented in the \pari{} command \texttt{solve(X=a,b,expr)}, which finds the real root of \texttt{expr} in the interval \( [a,b] \), where \texttt{expr(a)*expr(b) <= 0}}, to find 
	\[
	t \geq 2\,877\,166.69\ldots \,.
	\]
	Using \pari{}, one can easily enumerate the set \( S(2\,877\,167) \), and determine its size, as follows.
\begin{lstlisting}[language=gpin]
<@\inpr@> #[ D | D <- [1..2877167], D%24==23 && issquarefree(D)]
\end{lstlisting}\smallskip
\begin{lstlisting}[language=gpout]
<@\outpr@> 109337
\end{lstlisting}
	That is to say: for \( t_1 \coloneqq  2\,877\,167 \),  we have \( \# S(t_1) = 109\,337 \).  On the other hand, \autoref{eqn:classno} tells us that, \( h_{S(t_1)} \leq 9\,109.76\ldots \).  So we indeed  
	\[
	 12 h_{S(t_1)} + 2 \leq 109\,319.23\ldots < 109\,337 = \#S(t_1) \,.
	\]
	Hence we are guaranteed to find a linear dependence congruence relation modulo 4 amongst the \( 109\,337 \) elements of \( S(2\,877\,167) \).  This set is however essentially computationally intractable: already linear algebra (or the modulo 4 equivalent), is unfeasible for such a large system, especially if many solutions are expected.  In the next few paragraphs, we will refine this as much as possible.
	
	\paragraph{Improved estimate for \( t \)} The class number bound in \autoref{eqn:classno}, given by Griffin-Tsai-Ono is certainly not optimal.  We can instead compute the class number \( h(-D) \) directly and exactly using the command \texttt{qfbclassno(-D)} in \pari{}.  In particular, we can iterate through positive square-free \( D \equiv 23 \pmod{24} \), and track the maximum class number \( h_S \) encountered amongst all \( h(-D) \), until the inequality \( 12 h_S + 2 > \#S \) holds.  By doing so we can obtain obtain a significantly smaller candidate set satisfying the hypothesis of \autoref{thm:ono4}.  The following  \pari{} code excerpt accomplishes this.
\begin{lstlisting}[language=gpin]
<@\inpr@> hS=0; S=0; D=-1;
<@\inpr@> while(S <= 12*hS+2, D+=24; if(issquarefree(D), S++; hS=max(hS,qfbclassno(-D)); ) );
<@\inpr@> [D, hS, S]
\end{lstlisting}\smallskip
\begin{lstlisting}[language=gpout]
<@\outpr@> [315791, 1001, 12015]
\end{lstlisting}
	From the output in the final line, we learn that for \( t_2 \coloneqq 315\,791 = 24 \cdot 13\,158 - 1\), there are \( \#S(t_2) = 12\,015 \) many square-free positive \( D \equiv 23 \pmod{24} \leq t_2 \), and that the maximal class number amongst all of these discriminants is \( h_{S(t_2}) = 1\,001 \).  $\llbracket$One can check that \( h(-D) = 1\,001 \) occurs precisely for \( D = 312,311 = 24 \cdot 13\,013 - 1\), in the set \( S(t_2) \).$\rrbracket$
	
	Since \( 12 h_{S(t_2)} + 2 = 12\,014 < 12\,015 = \#S(t_2) \), the hypothesis of \autoref{thm:ono4} holds, and we are guaranteed to find a linear dependence congruence relation modulo 4 from the \( 12\,015 \) elements of \( S(315\,791) \).
	
	\paragraph{Further reductions to \( \#S \)}  By prioritising discriminants with smaller class number first, one can reduce the search space further.  We know from above that by all square-free \( D \equiv 23 \pmod{24} \) in the range \( 1 \leq D \leq 315\,791 \eqqcolon t_2 \), with class number \( h(-D) \leq 1\,001 \), we find a set satisfying the hypothesis of \autoref{thm:ono4}.  Instead let us select only \( h(D) \leq h_0 \), for certain smaller \( h_0 \), until we also find a valid set.  Again this can be accomplished with a short \pari{} program.
\begin{lstlisting}[language=gpin]
<@\inpr@> hList = [qfbclassno(-D) | D <- [1..315791], D%24==23 && issquarefree(D)];
<@\inpr@> hCnt = vector(vecmax(hList), h, #[ hD | hD <- hList, hD == h] );
<@\inpr@> h0 = 1; while(vecsum(hCnt[1..h0]) <= 12*h0+2, h0 = h0++);
<@\inpr@> [h0, vecsum(hCnt[1..h0])]
\end{lstlisting}\smallskip
\begin{lstlisting}[language=gpout]
<@\outpr@> [264, 3182]
\end{lstlisting}	
	From the output in the last line, we learn \( h_0 = 264 \) is the first time that selecting square-free positive \( D \equiv 23 \pmodi{24} \) with \( D \leq t_2 \) and \( h(-D) \leq h_0 \) gives a set satisfying the hypothesis of \autoref{thm:ono4}.  In particular, we obtain a set of 3\,182 elements, and indeed \( 12 \cdot 264 + 2 = 3\,170 \leq 3\,182 \).
	
	We have the following (abridged) table of class numbers and their multiplicities.
	\begin{table}[h!]
		\begin{tabular}{c|ccccccccccccccc}
			$ h $ & 3 & 5 & 7 & $\cdots$ & 100 & 101 & 102 & $\cdots$ & 200 & 201 & 202 & $\cdots$& 262 & 263 & 264 
			\\
			\hline
			$ \# \{ h(-D)=h \} $ & 1 & 1 & 1 & $\cdots$ & 14 & 6 & 11 & $\cdots$ & 36 & 8 & 13 & $\cdots$ & 23 & 10 & 52
		\end{tabular}
		\caption{Multiplicities of the class number \( h(-D) = h \), for square-free \( D \equiv 23 \pmodi{24} \) with \( 1 \leq D \leq 315\,791 \).}
		\label{tbl:classno}
	\end{table} \vspace{-1em}
	\begin{Rem}
		Although the class number problem has been completely solved for \( h(-D) \leq 100 \) by Watkins \cite{watkins2004classno}, (no doubt with further progress in the interim) we cannot easily exclude there being further square-free \( D \equiv 23 \pmodi{24} \), with \( D \geq t_2 \coloneqq 315\,791 \) and \( h(-D) \leq 264 \).  However, a direct computation for square-free \( D \equiv 23 \pmodi{24} \) in the range \( t_2 < D \leq 10^7 \) suggests all further class numbers are \( h(-D) \geq 271 \).
\begin{lstlisting}[language=gpin]
<@\inpr@> vecmin([qfbclassno(-D) | D <- [315792..10^7], D%24==23 && issquarefree(D)])
\end{lstlisting}\smallskip
\begin{lstlisting}[language=gpout]
<@\outpr@> 271
\end{lstlisting}
	$\llbracket$This takes under 2 minutes to run, so could easily be extended somewhat further.$\rrbracket$  In particular, we did not lose anything by restricting our attention to \( D \leq t_2 \coloneqq 315\,791 \).
	\end{Rem}

	We hence obtain the following candidate for a set \( S' \) satisfying the hypothesis of \autoref{thm:ono4}, :
	\[
		S' \coloneqq \{ 1 \leq D \leq 315\,791  \mid \text{$D \equiv 23 \pmodi{24}$, $D$ square-free, $h(-D) \leq 264$} \} \,.
	\]
	The computations above established that \( \#S' = 3\,182 \), and \( h_{S'} \coloneqq \max\{ h(-D) \mid D \in S' \} = 264 \) (as this class number does occur).  
	
	\paragraph{Final reductions to $\#S$}
	Since \( 12 h_{S'} + 2 = 3\,170 \) and \( \#S' = 3\,182 \), one can actually remove 11 elements from \( S' \), while preserving the inequality \( 12 h_S + 2 < \#S \) required for \autoref{thm:ono4}.  These ought to be the 11 largest elements, which we find with the following \pari{} code.  $\llbracket$Finding the 12-th largest element gives us the new bound to search until.$\rrbracket$
\begin{lstlisting}[language=gpin]
<@\inpr@> Sp = [ D | D <- [1..315791], D%24==23 && issquarefree(D) && qfbclassno(-D) <= 264];
<@\inpr@> vecextract(Sp, "-12..-1")
\end{lstlisting}
\smallskip
\begin{lstlisting}[language=gpout]
<@\outpr@> [241943, 245303, 248543, 254087, 255623, 259823, 259967, 262607,
  263783, 268367, 274103, 288047]
\end{lstlisting}
	Therefore take
	\[
		S \coloneqq \{ 1 \leq D \leq 241\,943  \mid \text{$D \equiv 23 \pmodi{24}$, $D$ square-free, $h(-D) \leq 264$} \} \,.
	\]
	This has cardinality \( \#S = 3\,171 \), and again \( h_{S} = 264 \), since we have removed at most 11 of the 52 elements with \( h(-D) = 264 \), in \autoref{tbl:classno}.  (In particular already \( D = 29\,279 = 24 \cdot 1\,120 -1 \) has \( h(-D) = 264 \).)

	\medskip
	
	The computation is now at a potentially tractable size.  However, we must compute a large number of values of the partition function, which is still time-intensive.  In particular, we need to compute just under
	\[
		\#S \cdot (12 h_{S} + 2) \cdot \tfrac{1}{3} = 3\,350\,690 \,,
	\]
	values (the factor $\frac{1}{3}$ arising from the condition \(\gcd(12,m) = 1 \) implied by \( \chi_{12}(m) \) in \( P(D; q) \)).  The largest value we need is
	\[
		p\Big( \frac{D m^2 + 1}{24} \Big) = p(101\,238\,639\,001) = 43744138\cdots\text{\texttt{\tiny 354\,428 digits omitted}}\cdots19363129 \equiv 1 \pmod{4}  \,,
	\]
	arising from \( D = \max(S) = 241\,943 = 24 \cdot 10\,081 - 1 \), and \( m = 12 h_{S} + 1 = 3\,169 \).  In \pari{} this computation takes approximately 20 minutes:
\begin{lstlisting}[language=gpin]
<@\inpr@> numbpart(101238639001);
<@\inpr@> ## \\ displays the runtime of the last command
\end{lstlisting}
\smallskip
\begin{lstlisting}[language=gpoutmsg]
***   last result: cpu time <@\color{olive}20min, 20,028 ms@>, real time <@\color{olive}20min, 20,059 ms@>.

\end{lstlisting}
	The same computation in \sage{} \cite{sage}, which calls the relevant routine in the optimised \texttt{C} library \texttt{FLINT} \cite{flint} takes under 2 seconds:
\begin{lstlisting}[language=sagein]
<@\inprsage@> %time _ = sage.libs.flint.arith.number_of_partitions(101238639001);
\end{lstlisting}
\smallskip
\begin{lstlisting}[language=sageout]
<@\outprsage@>CPU times: user 1.25 s, sys: 7.79 ms, total: 1.26 s
Wall time: 1.26 s
\end{lstlisting}
	The computation of all necessary partition values
	\[
		\Big\{ p\Big(\frac{Dm^2+1}{24}\Big) \pmodi{4} \Bigm\lvert  D \in S \,, 1 \leq m \leq 3\,169 \,, \gcd(m,12) = 1 \Big\} \,,
	\]
	in parallel with 160 threads on an \texttt{Intel(R) Xeon(R) Gold 6148 CPU @ 2.40GHz} machine has a runtime of around 2.5 hours in \sage{}.  After precomputing the necessary values modulo 4 in \sage{} first, we perform the rest of the search in \pari{} (which handles power series, and the conversion to/from arithmetic modulo 4 more transparently).

	\section{Computational results and explicit identities} 
	\label{sec:explicit}

	Precomputation of the partition values takes 2.5 hours with 160 threads.  Computation of the twisted generating series, the Hilbert class polynomials, and the resulting holomorphic normalisations takes in total 2 minutes on 160 threads.  Finally solving of the linear system modulo 4 takes 13 minutes.  We obtain a solution space containing 2\,421 non-zero linear congruence relations modulo 4.
	
	The first identity \( I_1 \) in the solution space consists of 131 non-zero terms.  The second identity \( I_2 \) consists of 211 non-zero terms, however the combination \( I_1 + I_2 \) (coefficients taken modulo 4) is slightly shorter, consisting of 198 non-zero terms. This is the explicit form of \cref{thm:main2} that we present.

	\subsection*{Identity 1}  Introduce the following set of discriminants\medskip
	\begin{equation*}
	\tag{131 terms} 
	\mathclap{\begin{aligned}[c]
		S^{(1)} &=  \Bigl\{ 24 k - 1 \mid k \in \bigl\{ 
	\begin{alignedat}[t]{16} 
&&    2, &&    4, &&    6, &&    8, &&   10, &&   11, &&   20, &&   22, &&   23, &&   25, &&   32, &&   35, &&   48, &&   52, &&   53,\\
&&   54, &&   55, &&   56, &&   57, &&   58, &&   60, &&   61, &&   64, &&   68, &&   70, &&   73, &&   75, &&   78, &&   81, &&   85,\\
&&   86, &&   88, &&   91, &&   92, &&   94, &&   97, &&   98, &&  101, &&  102, &&  103, &&  105, &&  106, &&  107, &&  109, &&  111,\\
&&  112, &&  114, &&  115, &&  117, &&  121, &&  125, &&  129, &&  130, &&  132, &&  135, &&  137, &&  139, &&  143, &&  144, &&  147,\\
&&  152, &&  153, &&  155, &&  159, &&  160, &&  163, &&  164, &&  167, &&  168, &&  169, &&  171, &&  173, &&  177, &&  178, &&  180,\\
&&  181, &&  182, &&  184, &&  185, &&  186, &&  188, &&  189, &&  190, &&  191, &&  195, &&  196, &&  197, &&  198, &&  201, &&  204,\\
&&  208, &&  211, &&  214, &&  215, &&  216, &&  218, &&  220, &&  226, &&  227, &&  228, &&  229, &&  231, &&  232, &&  235, &&  241,\\
&& \:\: 242, && \:\:  247, && \:\:  252, && \:\:  254, && \:\:  255, &&\:\:   256, && \:\:  260, && \:\:  261, && \:\:  262, && \:\:  263, && \:\:  266, && \:\:  269, && \:\:  271, && \:\:  276, && \:\:  281,\\
&&  282, &&  283, &&  284, &&  286, &&  287, &&  291, &&  296, &&  298, &&  301, &&  304, &&  309\phantom{,}
	\mathrlap{\bigr\} \, \smash{\Bigr\}} \,.}  
	\end{alignedat}\hspace{-1em}
	\end{aligned}}
	\end{equation*}
	\medskip

	\noindent We have \( h_{S^{(1)}} \coloneqq \max\{ h(-D) \mid D \in S^{(1)} \} = 92 \), and \( \max(S^{(1)}) = 7415 = 24\cdot309 - 1 \).  Since the following linear combination vanishes modulo 4 to order \( q^{12h_{S^{(1)}} + 2} = q^{1106} \) already, Sturm's bound \cite{sturm1987congruence} implies it vanishes identically modulo 4:
	\[
	\sum_{D \in S^{(1)}} 2 \cdot \widehat{P}_{S^{(1)}}(D; q) \equiv 0 \pmod{4} \,.
	\]
	This is the explicit form of \cref{thm:main1}.
	
	\subsection*{Identity 2}  Introduce the following sets of discriminants\medskip
	\begin{align*}
	\tag{66 terms}
	 & \hspace{-6em} \begin{aligned}[c] T_1 &=  \bigl\{ 24 k - 1 \mid k \in \bigl\{ 
		\begin{alignedat}[t]{12} 
&&    1, &&    2, &&    3, &&   12, &&   13, &&   20, &&   21, &&   29, &&   32, &&   36, &&   67, \\
&&   69, &&   73, &&   75, &&   84, &&   95, &&  100, &&  103, &&  115, &&  120, &&  121, &&  132, \\
&&  133, &&  140, &&  143, &&  147, &&  160, &&  164, &&  165, &&  166, &&  167, &&  168, &&  176, \\
&&  177, &&  181, &&  185, &&  187, &&  188, &&  189, &&  190, &&  192, &&  195, &&  197, &&  200, \\
&&  207, &&  208, &&  210, &&  211, &&  214, &&  215, &&  218, &&  219, &&  221, &&  225, &&  228, \\
&& \:\: 231, && \:\: 239, && \:\: 248, && \:\: 250, && \:\: 255, && \:\: 270, && \:\: 276, && \:\: 291, && \:\: 302, && \:\: 305, && \:\: 312\phantom{,}
\mathrlap{\bigr\} \, \smash{\Bigr\}} \,,} 
		\end{alignedat}
		\end{aligned}
	\\[1em]
	\tag{55 terms}
	& \hspace{-6em}  \begin{aligned}[c] T_2 &= \bigl\{ 24 k - 1 \mid k \in \bigl\{ 
	\begin{alignedat}[t]{12} 
&&    4, &&    9, &&   15, &&   22, &&   27, &&   31, &&   43, &&   51, &&   59, &&   60, &&   63, \\
&&   66, &&   68, &&   78, &&   79, &&   87, &&  101, &&  102, &&  107, &&  108, &&  110, &&  111, \\
&&  112, &&  113, &&  118, &&  119, &&  126, &&  139, &&  141, &&  144, &&  151, &&  152, &&  154, \\
&&  159, &&  161, &&  170, &&  172, &&  178, &&  183, &&  184, &&  193, &&  209, &&  212, &&  216, \\
&& \:\: 236, && \:\: 242, && \:\: 246, && \:\: 247, && \:\: 260, && \:\: 262, && \:\: 269, && \:\: 272, && \:\: 284, && \:\: 296, && \:\: 298\phantom{,}
\mathrlap{\bigr\} \, \smash{\Bigr\}} \,,} 
	\end{alignedat}
	\end{aligned}
	\\[1em]
	\tag{77 terms}
	& \hspace{-6em}  \begin{aligned}[c] T_3 &= \bigl\{ 24 k - 1 \mid k \in \bigl\{ 
	\begin{alignedat}[t]{12} 
&&    5, &&    6, &&    7, &&   10, &&   11, &&   16, &&   23, &&   26, &&   28, &&   33, &&   35, \\
&&   37, &&   42, &&   44, &&   48, &&   52, &&   54, &&   65, &&   70, &&   72, &&   77, &&   80, \\
&&   83, &&   86, &&   89, &&   91, &&   93, &&   97, &&  104, &&  105, &&  109, &&  114, &&  122, \\
&&  125, &&  127, &&  128, &&  129, &&  130, &&  131, &&  136, &&  137, &&  138, &&  142, &&  155, \\
&&  157, &&  169, &&  173, &&  175, &&  180, &&  182, &&  198, &&  202, &&  203, &&  205, &&  213, \\
&&  217, &&  220, &&  222, &&  226, &&  227, &&  229, &&  232, &&  235, &&  244, &&  251, &&  252, \\
&& \:\: 254, && \:\: 257, && \:\: 266, && \:\: 271, && \:\: 278, && \:\: 281, && \:\: 282, && \:\: 286, && \:\: 289, && \:\: 301, && \:\: 309\phantom{,}
\mathrlap{\bigr\} \, \smash{\Bigr\}} \,.} 
	\end{alignedat}
	\end{aligned}
	\end{align*}\medskip
	
	\noindent Write \( S^{(2)} = T_1 \cup T_2 \cup T_3 \).  Then we have \( h_{S^{(2)}} \coloneqq \max\{ h(-D) \mid D \in S^{(2)} \} = 92 \), and  \( \max( S^{(2)} ) = 7487  = 24\cdot312 - 1 \).  Since the following linear combination vanishes modulo 4 to order \( q^{12h_{S^{(2)}} + 2} = q^{1106} \) already, Sturm's bound \cite{sturm1987congruence} implies it vanishes identically modulo 4:
	\[
		\sum_{D \in T_1} \widehat{P}_{S^{(2)}}(D; q) + \sum_{D \in T_2} 2 \cdot \widehat{P}_{S^{(2)}}(D; q) + \sum_{D \in T_3} 3 \cdot \widehat{P}_{S^{(2)}}(D; q) \equiv 0 \pmod{4} \,.
	\]
	This is the explicit form of the linear dependence congruence relation modulo 4 given in \cref{thm:main2}.
	
	\appendix
	
	\section{\sage{} and \pari{} routines to verify \autoref{thm:main-schematic}}
	\label{sec:code}
	
	The following code precomputes the partition values necessary, and finds all relations amongst the holomorphic normalisations \( \widehat{P}_S(D; q) \) for \( S = \{ 1 \leq D \leq 24\cdot350-1 \mid \text{$D \equiv 23 \pmodi{24}$, square-free} \} \).  In particular the relations in \cref{thm:main1,thm:main2} are (essentially) the first two solution vectors.  The claim in \cref{thm:main3} follows by listing which discriminant gives the last non-zero entry in each non-zero solution vector.\medskip
	
	The \sage{} code in file \texttt{sage\_part350.sage} is used to precompute the necessary partition values, it generates the file \texttt{values\_part350.txt} included with the arXiv submission.  Runtime is approximately 45 minutes--1 hour on a single thread.\medskip

\begin{lstlisting}[language=Python,style=sage,captionpos=t,caption={{\sage{} code to compute partition values necessary to find relations for \( S = \{ 1 \leq D \leq 24\cdot350-1 \mid \text{$D \equiv 23 \pmodi{24}$, square-free} \} \).  File: \texttt{sage\_part350.sage}}\\}]
# use FLINT to compute partition numbers
from sage.libs.flint.arith import number_of_partitions

# class number computation is faster via gp/pari
def qfbclassno(d): return int(gp.qfbclassno(d));
def part(x): return (number_of_partitions(x) % 4);

# Search discriminants up to Dmax
Dmax = 24*350-1;
S = [d for d in range(1, Dmax+1) if d%24 == 23 and is_squarefree(d)];
hS = max([qfbclassno(-d) for d in S]);

# For this Dmax, hS = 134, and so compute to precision q^1620, 
# slightly more than needed by the Sturm bound, for certainty
bound = 12*hS+2 + 10;

pargs = [ (d*m^2+1)/24 for d in S for m in range(1, bound+1) if gcd(m,12) == 1];

# compute the partition values and output as a string with gp/pari syntax
pvals = [ [n, part(n)] for n in sorted(pargs) ];
result = "parts = " + str(pvals) + ";";

# write this to the a file for use with gp
outfile = open("values_part350.txt", "w");
outfile.write(result);
outfile.close();
\end{lstlisting}
	\medskip

	The main computation occurs via \pari{} in file \texttt{pari\_part350.gp}.  Precomputed partition values are loaded from \texttt{values\_part350.txt}.  Generation of holomorphic normalisations, and solving the resulting linear system modulo 4, both take 3--4 minutes.  For discriminants up to \( 24\cdot350-1 \), there are 15 non-zero solutions.  These solutions are accessible in the variable \texttt{solnvec[,i]}, for \( 1 \leq i \leq 15 \), and can be displayed ``nicely'' via \texttt{display(solnvec[,i])}.

\begin{lstlisting}[language=gp,captionpos=t,caption={{\pari{} code to find relations amongst the series \( \widehat{P}_S(D;q) \), for \( S = \{ 1 \leq D \leq 24\cdot350-1 \mid \text{$D \equiv 23 \pmodi{24}$, square-free} \} \).  File: \texttt{pari\_part350.gp}}\\}]
\\ load and process the precomputed partition values
\\ stored in variable "parts" in the given file
read("values_part350.txt");
parts = vecsort(parts);

\\ return precomputed partition values by searching for the argument
pargs = [p[1] | p <- parts];
pvals = [p[2] | p <- parts];
part(n) = { if(frac(n) != 0, 0, 
    i = vecsearch(pargs, n);
    if(i==0, error("partition not computed"), pvals[i]);
  );    }

\\ now search discriminants up to Dmax
Dmax = 24*350 - 1;
S = [D | D <- [1..Dmax], D%24==23 && issquarefree(D)];
hS = vecmax(apply(D -> qfbclassno(-D), S));

\\ For this Dmax, hS = 134, and so compute to precision q^1620, 
\\ slightly more than needed by the Sturm bound, for certainty
bound = 12*hS+2 + 10;

\\ compute Delta and Delta^-1 mod 4 to the necessary precision
deltaser = sum(i=1, bound, (ramanujantau(i) * Mod(1,4))*q^i, q*O(q^bound));
deltainvser = 1/deltaser;

\\ routine to convert a power-series to a list of coefficients
serVec(s) = Vec(s, -serprec(s, variable(s)));

\\ displays useful statistics and the D = 24k-1 having coeff  1, 2 or 3,
\\ [#terms, [#1's, #2's, #3's], [k with 1's, k with 2's, k with 3's]]
display(vin) = {   v = vin % 4; proc = [[],[],[]];
  for(i=1, #v, 
    if(v[i] != 0, proc[v[i]] = concat(proc[v[i]], (S[i]+1)/24 ));
  );
  [#concat(proc), apply(i->#i, proc), proc];    }

\\ return just the list of nonzero k's appearing
nonzero(vin) = vecsort(concat(display(vin)[3]));

\\ compute the P(D; q) series
Pd(D) = {   sum(m=1, bound,
    part((D*m^2 + 1)/24) * kronecker(12, m) * 
    sum(n=1, bound \ m , kronecker(-D, n)  * q^(m*n) ),
   q*O(q^bound) );    }

\\ compute the holomorphic normalisation
holPd(D) = Pd(D) * deltaser^hS * subst(polclass(-D), x, deltainvser);

\\ compute all series, and extract integer matrix of coefficients 
ser = apply(holPd, S);
eqns = Mat(apply(i -> lift(serVec(i + q*O(q^bound)))~, ser));

\\ solve  "eqns*vec = 0 modulo 4", and return full solution space in result[2]
result = matsolvemod(eqns, 4, 0, 1) % 4;
solvecs = select(i -> vecmax(i)!=0, result[2]);

\\ the first solution in Theorem 1(i) and Identity 1
display(solvecs[,1])

\\ the second solution
display(solvecs[,2])

\\ the simpler combination in Theorem 1(ii) and Identity 2
display(solvecs[,1] + solvecs[,2])

\\ the k's in D = 24k-1 which give new identities in Theorem 1(iii)
vector(#solvecs, i, vecmax( nonzero(solvecs[,i]) ))
\end{lstlisting}

	\bibliographystyle{habbrv2} 
	\bibliography{bibliography.bib}

\end{document}